\documentclass[11pt]{article}
\usepackage{fullpage}
\usepackage{amsthm}
\usepackage{amssymb}
\usepackage{amsmath}
\usepackage{graphicx}
\usepackage{tikz}
\usepackage{mathrsfs}
\usepackage{float}
\usepackage{makecell}
\usepackage{amscd}
\usepackage{multirow}
\usepackage{hhline}
\usepackage{verbatim}

\newcommand{\E}{\mathbb{E}}

\newcommand{\R}{\mathbb{R}}

\renewcommand{\epsilon}{\varepsilon}

\DeclareMathOperator{\Span}{span}

\usepackage{relsize}
\newtheorem{theorem}{Theorem}
\newtheorem{lemma}[theorem]{Lemma}
\newtheorem*{unnumtheorem}{Theorem}

\newtheorem{remark}{Remark}

\newtheorem{corollary}[theorem]{Corollary}
\newtheorem{claim}[theorem]{Claim}
\theoremstyle{definition}

\title{On the complexity of random polytopes} 
\author{Andrew Newman\thanks{Technische Universit\"{a}t Berlin}}

\date{\today}

\begin{document}
\maketitle
\abstract
There are (at least) two reasons to study random polytopes. The first is to understand the combinatorics and geometry of random polytopes especially as compared to other classes of polytopes, and the second is to analyze average-case complexity for algorithms which take polytopal data as input. However, establishing results in either of these directions often requires quite technical methods. Here we seek to give an elementary introduction to random polytopes avoiding these technicalities. In particular we explore the general paradigm that polytopes obtained from the convex hull of random points on a sphere have low complexity.

\section{Introduction}
By now, random polytopes obtained by sampling points uniformly from a fixed convex body $K$ have been extensively studied. In one research direction they have been studied from the point of view of understanding the expected values of particular functionals such as volumes and $f$-vectors, see for example \cite{BFV, Barany89, BMT, Raynaud, ReitznerBoundary, SW20032, SW2003}. On the other hand, Borgwardt, see \cite{Borgwardt, GiftWrapping, BeneathBeyond}, has studied random polytopes in the situation of establishing average-case analysis of convex-hull and linear-programming algorithms. In both situations, combinatorial complexity of the polytopes themselves and complexity of algorithms on random polytopes, however, one finds that random polytopes are much nicer than general polytopes can be. Indeed, a polytope on $m$ vertices in $n$ dimensions can have as many as $m^{\lfloor n/2 \rfloor}$ faces. These polytopes arise as the cyclic polytopes, see for example \cite{Ziegler}. Polytopes with high combinatorial complexity naturally increase the algorithmic complexity of, for example, computing the facet description from the vertices. On the other hand random polytopes, as we will see, do not have many faces in terms of their vertices, and this is key to analysis in \cite{Borgwardt, GiftWrapping, BeneathBeyond}.\\

Now, full analysis of random polytopes can be quite technical, especially in fairly general settings. The technicalities often involve evaluating rather complicated integrals. Of course these types of integrals are unavoidable if one wishes to establish the precise answers in general. Nevertheless, some of the intuition can be lost in the process of these necessary technicalities. Here we try and preserve the intuition and present proofs about the complexity of random polytopes for an audience familiar with probability and polytopes, but likely unfamiliar with random polytopes. Unsurprisingly, our results are not as good as the known results, but we show that in the asymptotic case the more geometric approach presented here really doesn't lose too much on the true complexity given in the existing literature. \\

The model discussed here will be denoted as $P(n, m)$ and will be the probability space of polytopes in $\R^n$ where $P$ is sampled from $P(n, m)$ by taking the convex hull of $m$ points chosen uniformly at random from the boundary of the sphere in $\R^n$. The probability space is at the intersection of the work of Borgwardt in \cite{Borgwardt, GiftWrapping, BeneathBeyond}, who considered the family of all rotationally-symmetric probability distributions, and \cite{BFV, Barany89, BMT, Raynaud, SW20032, SW2003} who consider sampling points uniformly from a fixed convex-body or boundary of a fixed convex body. Here we discuss three problems about $P \sim P(n, m)$. We first consider the purely geometric problem of determining the expected complexity of $P \sim P(n, m)$ (as in \cite{BMT}).  Next, we turn our attention to the related problem of the simplex method applied to $P$. This problem is discussed extensively in Borgwardt's book \cite{Borgwardt}, and is important in an effort to understand why the simplex method works well in practice even though it is known due to \cite{KleeMinty} to have exponential complexity in the worse case. Finally we discuss convex hull algorithms on $P$, that is the problem of determining the facets of $P$ from its vertices. We conclude with an appendix showing computations for some relevant geometry quantities in the interest of keeping this paper self-contained. 

\section{Bounding the complexity of a random polytope}
Key to much of the work on average case analysis is that rather than directly bounding the complexity of an algorithm applied to random polytopes, one can bound the complexity of the input to the algorithm, that is the random polytope itself. Here we take as an example the problem of determining the expected number of facets of a random polytope determined by sampling $m$ points uniformly from the unit sphere in $\R^n$. This is known due to \cite{BMT} who gives a precise formula in terms of $n$ and $m$. Indeed \cite{BMT} gives following asymptotic result about the number of facets of $P \sim P(n, m)$.

\begin{unnumtheorem}[\cite{BMT}]
Fix $n \geq 2$, $f_{n - 1}(P)$ for $P \sim P(n, m)$ satisfies
$$\lim_{m \rightarrow \infty} \frac{\E(f_{n - 1}(P))}{m} = \frac{2}{n} \gamma_{(n - 1)^2} \gamma_{n - 1}^{-(n - 1)}$$
with $\left\{\gamma_k\right\}_{k = 0}^{\infty}$ given by the recurrence $\gamma_0 = \frac{1}{2}$ and $\gamma_{k + 1} = \frac{1}{2 \pi (k + 1) \gamma_k}$. 
\end{unnumtheorem}

The first few values of $F_n := \frac{2}{n} \gamma_{(n - 1)^2} \gamma_{n - 1}^{-(n - 1)}$ are $F_2 = 1, F_3 = 2, F_4 = \frac{24 \pi^2}{35}, F_5 =  \frac{286}{9}, F_6 = \frac{1{,}296{,}000 \pi^4}{676{,}039}$, and in the limit $F_n$ grows exponentially in $n$. 
The result of \cite{BMT} is obtained by setting up and evaluating the right integrals. In section \ref{highlevel} we overview the integral one would set up, however we prove the following result without having to evaluate such an integral. It is of course a weaker result than the result of \cite{BMT}, but not so much weaker and the proof comes more directly from the geometry of the sphere.

\begin{theorem}\label{spherecomplexity}
Fix $n \geq 2$, then the expected number of facets of $P \sim P(n, m)$ as $m$ tends to infinity is at most $C_n (\log m)^{n - 1} m$ where $C_n$ is a constant depending only on $n$. 
\end{theorem}

\subsection{High-level view of the computation} \label{highlevel}
In this section we overview the integral one would set up to compute the expected number of facets in $P \sim P(n, m)$. However, rather than evaluating this integral as in \cite{Borgwardt} or \cite{BMT}, we will estimate it from above using the geometry of the $(n-1)$-dimensional sphere. Under the model $P(n, m)$, we may assume that our vertices are in general position, and thus any set of $n$ points from our set of $m$ points $a_1, ..., a_m$ potentially determines a facet. By linearity of expectation, it suffices to compute the probability that $a_1, ..., a_n$ determines a facet. Indeed if we denote this probability by $p_n$ then clearly the expected number of facets is $\binom{m}{n} p_n$. \\

Suppose first that $a_1, ..., a_n$ are fixed, and we wish to bound the probability that these $n$ vertices determine a facet. Let $H(a_1, ..., a_n)$ denote the hyperplane spanned by $a_1, ..., a_n$. Now $H(a_1, ..., a_n)$ determines a facet of $P(n, m)$ if and only if all other $m - n$ points lie on the same side of $H(a_1, ..., a_n)$ as each other. Of course as $m$ tends to infinity the probability that $P \sim P(n, m)$ contains the origin in its interior converges quickly to 1, thus the probability that $H(a_1, ..., a_n)$ determines a facet of $P$ is asymptotic to the probability that all other $m - n$ points lie on the same side of $H(a_1, ..., a_n)$ as the origin.  Denote this half-space by $H^-(a_1, ..., a_n)$ and the other half-space by $H^+(a_1, ..., a_n)$. The next task will be to determine the probability that a randomly chosen point on $S^{n - 1} \subseteq \R^n$ lies in $H^+(a_1, ..., a_n)$. \\

For the hyperplane $H(a_1, ..., a_n)$, we let $h(a_1, ..., a_n)$ denote its distance from the origin. This notation give us a very convenient description of $S^{n - 1} \cap H^+(a_1, ..., a_n)$ in the following claim. 
\begin{claim}
$S^{n - 1} \cap H^+(a_1, ..., a_n)$ is an $n$-dimensional spherical cap of height $1 - h(a_1, ..., a_n)$. 
\end{claim}
The proof of this claim is immediate by using rotational symmetry of the random model to rotate the given hyperplane so that the vector normal to it is the vector $h e_n$. \\

Now for any $h$, we define $SA_n(h)$ to be the $(n-1)$-dimensional Lebesgue measure of the $(n-1)$-dimensional spherical cap of height $h$. We also let $s_n$ denote the $(n-1)$-dimensional Lebesgue measure (the hyper-surface area) of the unit sphere in $\R^n$. For a fixed set of $n$ points (in general position) $a_1, ..., a_n$, the probability that a randomly chosen point on the $n$-dimensional unit sphere lies in the $H^+(a_1, ..., a_n) \cap \omega_n$ is 
$$\frac{SA_n(1 - h(a_1, ..., a_n))}{s_n}.$$

From this it follows that the following integral computes the probability that $a_1, ..., a_n$ determines a facet of $P \sim P(n, m)$
$$\int_{(a_1, ... , a_n) \in S^{n - 1}} \left(1 - \frac{SA_n(1 - h(a_1, a_2, ..., a_n))}{s_n} \right)^{m - n} da_1 da_2, ... da_n$$

We refer the reader to for example \cite{BeneathBeyond} to see the evaluation of this integral. Our purpose here is to use the idea which sets up this integral to estimate the complexity of $P \sim P(n, m)$. \\

 Essentially there are two (random) properties of $H(a_1, ..., a_n)$ that we will consider. The first is the distance from $H(a_1, ..., a_n)$ from the origin, that is the value of $h(a_1, ..., a_n)$, and the second is whether or not $H(a_1, ..., a_n)$ induces a facet. Notice that the first of these properties depends on $a_1, ..., a_n$, and the second depends on $a_{n + 1}, ..., a_m$. These events are not independent from one another however. Indeed the further $H(a_1, ..., a_n)$ is from zero, the less space it leaves for a point among $a_{n + 1}, ..., a_m$ to appear on the other side of the hyperplane. Thus the further $H(a_1, .., a_n)$ is from zero, the more likely it is to induce a facet.\\

As we consider these two events, our argument will split into two steps. In the first step we show that if $H(a_1, ..., a_n)$ is close to the origin (which here will mean that $h(a_1, ..., a_n) < 1 - \delta$ for some $\delta$ depending on $m$ and $n$), then the probability that $H(a_1, ..., a_n)$ induces a facet of $P \sim P(n, m)$ is at most $o(m^n)$. This implies that the expected number of such facets is at most $m^n (o(m^n)) = o(1)$. In other words, with high probability the Hausdorff distance of $P$ to $S^{n - 1}$ is at most $\delta$. Essentially this means that almost all of the hyperplanes contribute nothing to the random polytope. \\

In the second step we handle hyperplanes that are far from the origin, that is hyperplanes the cut off a small spherical cap. In this case $a_1, ..., a_n$ all have to be close to one another. It is easy to verify that if the distance (in $\R^n$) from $a_i$ to $a_j$ is at least $\epsilon$ then the midpoint of the segment through $a_i$ and $a_j$ has norm at most $1 - \epsilon^2/8$ and that is an upper bound for the distance from the origin to any hyperplane containing $a_i$ and $a_j$. Thus, by the choice of distribution which defines $P(n, m)$ one should expect few hyperplanes defined by vertices which are close together. \\

We now make this argument precise.


\subsection{Proof of Theorem \ref{spherecomplexity}}
\begin{proof}[Proof of Theorem \ref{spherecomplexity}]
As described above we seek $\delta = \delta(n, m)$ so that we may say that asymptotically almost surely no facet of $P$ is within distance $1 - \delta$ of the origin, and so that there are only a few possible facets within distance $\delta$ of the boundary. \\

With foresight define $\delta = \delta(n, m)$ to be such that 
$$\frac{SA_n(\delta)}{s_n} = \frac{2(n + 1)\log m}{m}.$$

Now, if $a_1, ..., a_n$ satisfies $h := h(a_1, ..., a_n) \leq 1 - \delta$, then the probability that $H(a_1, ..., a_n)$ determines a hyperplane is at most
$$G(h)^{m - n} = (1 - SA_n(1 - h)/s_n)^{m - n} \leq (1 - SA_n(\delta)/s_n)^{m - n} \leq (1 - 2(n + 1)\log m/m)^{m - n}$$
For $m \geq 2n$, this is at most $m^{-(n + 1)}$. Thus taking a union bound over all $\binom{m}{n}$ possible facets, we have that the expected number of hyperplanes which cut off a spherical cap of height greater than $\delta$ and induce a facet of $P \sim P(n, m)$ is at most $m^{-1} = o(1)$.\\

The above implies that with probability at most $1/m$ all facets of $P(n, m)$ are within Hausdorff distance $\delta$ of the boundary of the unit ball. (Note that the Hausdorff distance between two sets $A$ and $B$ is defined to be the maximum over all points $a \in A$ of the distance from $a$ to $B$. In the present setting this means that $P$ is entirely contained in the spherical shell obtained by removing the $n$-dimensional ball centered at the origin of radius $1 - \delta$ from the $n$-dimensional unit centered at the origin.) Intuitively, this means that all facets are small. Therefore we will upper bound the number of facets by finding an upper bound on the number of small facets. To do so, we will consider the local picture. Fix a vertex $a_1$. For any $a_2, ..., a_n$, we bound the probability that $a_2, ..., a_n$ are all close enough to $a_1$ that the facet induced is small. \\

The first task is to quantify ``close enough" and, by extension, what we mean by a "small facet" quantitatively. If the facet induced by $a_1, ..., a_n$ is to be within Hausdorff distance $\delta$ of the unit sphere, and all vertices are on the unit sphere then $a_2, ..., a_n$ must all lie within a spherical cap of height at most $4 \delta$ centered at $a_1$. This follows by the Pythagorean theorem, and we refer the reader to the appendix for the details. Now for fixed $a_1$ we bound the probability that $a_2, ..., a_n$ are all within the spherical cap of height $4 \delta$ centered at $a_1$.\\

In the first step, we did not need to know anything about the behavior of the function $SA_n(y)$. Here however, we do but only insofar as we will need bounds when $y$ is small as $\delta$ tends to zero as $m$ grows.  In the appendix we show that there exists a constant $C$ depending on $n$ so that for $y$ small enough, $2C y^{(n - 1)/2} \geq SA_n(y) \geq C y^{(n - 1)/2}$. It follows that for $m$ large enough we have 
\begin{eqnarray*}
\frac{SA_n(4 \delta)}{s_n} &\leq& \frac{2C(4 \delta)^{(n - 1)/2}}{s_n} \\
&\leq& \frac{2(4^{(n - 1)/2}) C \delta^{(n - 1)/2}}{s_n} \\
&\leq& \frac{2(4^{(n - 1)/2}) SA_n(\delta)}{s_n} \\
&=& \frac{16(4^{(n - 1)/2}) (n + 2)\log m}{m}
\end{eqnarray*}
Thus the probability that $a_2, ..., a_n$ all lie in this spherical cap is at most 
$$\left(\frac{16(4^{(n - 1)/2}) (n + 2)\log m}{m} \right)^{n - 1}$$
Multiplying over all choices of $a_2, ... , a_n$ gives us that the expected number of facet containing $a_1$ given that no facets are within distance $1 - \delta$ of the origin is at most
\begin{eqnarray*}
\binom{m - 1}{n - 1} \left(\frac{16(4^{(n - 1)/2}) (n + 2)\log m}{m} \right)^{n - 1} &\leq& C_n \log^{n - 1} m,
\end{eqnarray*}
where $C_n$ is a constant that depends only on $n$. Thus taking a sum over all $m$ vertices gives the result. 
\end{proof}

\begin{remark}
Our proof actually shows something stronger than just that the expected number of facets is $O(m \log^{n - 1} m)$, it in fact bounds the probability that the number of facets exceeds this bound quite well. Indeed the first step shows that with probability $O(1/m)$ the Hausdorff distance between $P$ and $S^{n - 1}$ is at most $\delta$ and the second step bounds the number of vertices that are close to a given vertex, however this is controlled by a binomial random variable with $m$ trials and probability $SA_n(4 \delta)/s_n$. Thus one has that with probability $1 - o(1/m)$ every vertex has at most $C_n \log m$ vertices in its link, and so we have the following result regarding the concentration of the complexity of $P \sim P(n, m)$. 
\end{remark}
\begin{theorem}\label{concentration}
For any fixed $n$ and $K$ there exists a constant $C = C(n, K)$ so that with probability at least $1 - O(1/m^K)$ the number of facets of $P(n, m)$ is at most $C m \log^{n - 1} m$. 
\end{theorem}
The proof of this theorem is exactly like the proof of Theorem \ref{spherecomplexity} with $\delta$ set (for $m$ large enough) so that
$$\frac{SA_n(\delta)}{s_n} = \frac{(K + n) \log m}{m}.$$




\section{The shadow-vertex algorithm on random polytopes}
Random polytopes are also interesting from the point of view of average case analysis of algorithms. On of the most ambitious instances of such analysis is K.H. Borgwardt's book \emph{The simplex method. A probabilistic analysis} \cite{Borgwardt}. In \cite{Borgwardt}, Borgwardt analyzes the shadow-vertex algorithm for linear programming and as the main result of his book proves the following about the average complexity.
\begin{theorem}[Borgwardt \cite{Borgwardt}]
For all distributions according to the rotation-invariant model, the shadow-vertex algorithm solves the complete linear programming problem in at most 
$$m^{1/(n - 1)}(n + 1)^4 \frac{2 \pi}{5} \left(1 + \frac{e\pi}{2} \right)$$
pivot steps on average.
\end{theorem}

Proving this theorem is quite technical with many clever techniques of integration required. One of the primary goals presently will be to, in a few pages, establish a result close to Borgwardt's for the uniform distribution on the $n$-dimensional sphere. Indeed we will prove the following asymptotic result in order to aide in understanding how the geometry of the sphere contributes to the low complexity of the algorithm without requiring the technical steps necessary for Borgwardt's full result. \\
\begin{theorem}\label{asymptoticshadowcomplexity}
Fix $n \geq 2$, then for $P \sim P(n, m)$ and $v$ a random vector on the unit sphere, the expected number of pivot steps required to solve the corresponding linear programming problem via the shadow-vertex algorithm is on average at most $\tilde{O}(m^{1/(n - 1)})$.
\end{theorem}
Note that we are using $\tilde{O}$ for the usual soft-$O$ notation, that is the upper bound is of the form $C m^{1/(n - 1)} (\log m)^k$ where $k$ and $C$ depend only on $n$. \\

Before describing the shadow-vertex algorithm in detail, recall that a linear programming problem is a problem of the following form for vectors $v, x, a_1, ...., a_m$ in $\R^n$ and $\langle \cdot, \cdot \rangle$ denoting the usual inner product on $\R^n$.
\begin{center}
Maximize $\langle v, x \rangle$ subject to $\langle a_1, x \rangle \leq 1$, ...., $\langle a_m, x \rangle \leq 1$. 
\end{center}
Note that we have normalized the problem by setting the right-hand sides to 1 and that we do not require that $x \geq 0$. Of course, the feasible region of this problem is a polytope with facets determined by $a_1, ..., a_m$ and the problem is to determine the vertex of this polytope closest to $v$. \\

As is often the case, it is convenient to consider the dual version of this problem. If $X$ is the polytope given by the feasible region of the problem above, then the dual polytope $Y$ is defined by be $\{y \in \R^n \mid \langle x, y \rangle \leq 1$ for all $x \in X\}$. It is easy to see that, in this case $Y$ is the convex hull of $0, a_1, ..., a_m$ (see for example Lemma 1.5 of \cite{Borgwardt}). One also observes that in the dual situation, the linear programming problem is to find the facet of $Y$ through which the ray $\R^{+} v$ passes. Just as in \cite{Borgwardt} we analyze the average complexity of the shadow-vertex algorithm by analyzing its dual description.  In particular this fits with the vertex-description we have used to define $P(n, m)$. \\

In the dual description, the approach of the shadow-vertex algorithm, first described by \cite{GassSaaty}, will be to start at dimension 2 and increase one dimension at a time, at each stage finding a new vertex $a_t$ which is a vertex of the optimal facets. At the end all vertices of this optimal facet will be determined. The ``shadow" portion of this algorithm comes from the fact that at dimension $k \geq 3$, we have the first $k - 1$ vertices of our optimal simplex and to find the next vertex we will run the usual simplex algorithm on the \emph{polygon} given by intersecting our $k$-dimensional polytope with some appropriately-chosen 2-dimensional plane. Doing this will give us the next vertex and allow us to move to $k + 1$ dimensions and repeat the process. Thus we are always considering some 2-dimensional shadow of our polytope. Before describing the entire algorithm, we will prove the following lemma from which Theorem \ref{asymptoticshadowcomplexity} will follow almost immediately.

\begin{lemma}\label{shcomplexity}
Fix $n \geq 2$, and $u$ and $v$ nonparallel vectors in $\R^n$. Let $P \sim P(n, m)$ and let $P'$ denote the intersection of $P$ with the plane $\Span(u, v)$, then the expected number of edges of $P'$ is $\tilde{O}(m^{1/(n - 1)})$. 
\end{lemma} 
Before presenting the proof completely, we sketch the idea. First of all, the edges of $P'$ are determined by the facets of $P$ which are intersected by $\Span(u, v)$. Clearly, a facet $\sigma$ contributes an edge to $P'$ if and only if $\sigma \cap \Span(u, v)$ is nonempty. Moreover, by non-degeneracy, the intersection of a facet at $\Span(u, v)$ is a line segment or empty. Thus every facet contributes at most one edge to $P'$. \\

Now as we showed above, for $P \sim P(n, m)$ we have that the expected degree of every vertex is $\tilde{O}(1)$ indeed it is $O(\log^{n - 1}(m))$ as in the proof of Theorem \ref{spherecomplexity}. Moreover, the proof of this fact comes from showing that with high probability the facets that contain any fixed vertex $a_i$ have all of their vertices near $a_i$. Thus we actually have that all the edges of $P$ are short. It follows therefore that if we take the 2-dimensional disk given by the intersection of the $n$-dimensional ball and $\Span(u, v)$, we have that vertices that are too far from this intersection cannot contribute any of their cofacets to $P'$. Thus the problem becomes to bound the number of vertices close to the disk. It turns out that it will ultimately suffice to bound the number of vertices which \emph{project} near the boundary of the 2-dimensional disk given by the intersection. This is illustrated in Figure \ref{figcircle}. Once we upper bound this (by $\tilde{O}(m^{1/(n - 1)})$) we can multiply by the $\tilde{O}(1)$ bound on the maximum vertex degree to upper bound the number of facets that contribute an edge to $P'$. 
\begin{center}
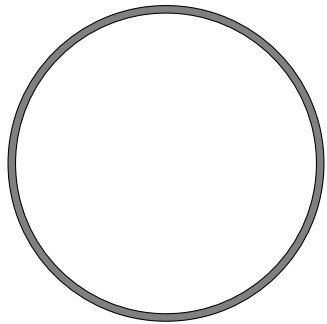
\begin{figure}[H]\label{figcircle}
\centering
\begin{tikzpicture}

    \fill[gray] (0, 0) circle (2.1);
    \fill[white] (0, 0) circle(2);
        \draw (0, 0) circle (2);
    \draw (0, 0) circle (2.1);
\end{tikzpicture}
\caption{We intersect $P$ with a 2-dimensional plane. Above we see the intersection of the plane with the $n$-dimensional ball. Only those vertices which \emph{project} close to the boundary (in the gray region) contribute to the complexity of $P'$, thus the task becomes to estimate the number of such vertices.}
\end{figure}
\end{center}

\begin{proof}[Proof of Lemma \ref{shcomplexity}]
By rotational symmetry of $P(n, m)$ we may assume without loss of generality that $u = (1, 0, ..., 0)$ and $v = (0, 1, ..., 0)$. Now a facet $\sigma$ of $P$ will determine an edge of $P'$ if and only if $\sigma \cap \Span(u, v)$ is nonempty. By nondegeneracy, which we may assume since the points $a_1, ..., a_m$ are chosen independent of $u$ and $v$, we have that $\sigma$ contributes at most one edge to $P'$. The goal will be to upper bound the number of such facets. \\

Of course, with high probability, the maximum degree of a vertex in $P \sim P(n, m)$ is $\tilde{O}(1)$, so in this case it will suffice to count the number of vertices that belong to a facet that contributes an edge to $P'$. Indeed one can count the number of such vertices and multiply by the degree bound to overestimate the number of edges in $P'$. \\

Now that we have reduced the problem to enumerating vertices contributing cofacets to $P'$, we prove that with high probability the edges of $P$ are short, and thus verices that are too far away from the intersection plane cannot contribute a facet to be an edge of $P'$. More precisely we will set $\epsilon$ and $\Delta$ and condition on the event that the following two events hold for $P$.
\begin{enumerate}
\item Every edge of $P$ has length at most $\epsilon$, and
\item No vertex of $P$ is contained in more than $\Delta$ facets of $P$.
\end{enumerate}
Now $\epsilon$ and $\Delta$ will depend on $m$ and will come from the geometry of the $n$-dimensional ball. Let $C_1$ and $C_2$ be large enough constants (depending on $n$, but not on $m$) and let $$\epsilon = C_1\left(\log m/m \right)^{1/(n - 1)}$$ and $$\Delta = C_2 (\log m)^{n - 1}.$$ 

We first bound the probability that $P$ has an edge of length greater than $\epsilon$. Observe that if $a_1$ and $a_2$ are connected by an edge of $P$ and that edge has length at least $\epsilon$, then the distance from the origin to the line through $a_1$ and $a_2$ is at most $1 - \frac{\epsilon^2}{8}$, by the Pythagorean theorem. Thus, any hyperplane containing $a_1$ and $a_2$ is at least as close to the origin, so it suffices to bound the probability that $P$ has such a facet. \\

Suppose that the distance from $0$ to the hyperplane through $a_1, a_2, ..., a_n$ is at most $1 - \frac{\epsilon^2}{8}$, then the spherical cap on the other side of zero determined by this hyperplane has height at least $\frac{\epsilon^2}{8}$. Such a spherical cap has suface area at least $C_3(\frac{\epsilon^2}{8})^{(n - 1)/2}$ for some constant $C_3$ depending on $n$ (see Corollary \ref{geometrycorollary}). Thus the probability that it induces a facet is at most
$$\left(1 - C_3\left(\frac{\epsilon^2}{8}\right)^{(n - 1)/2}\right)^{m - n} \leq \exp\left(-(1 - o_m(1))C_4 m \epsilon^{n - 1}\right) = \exp(-(1 - o_m(1)) C_4 C_1 \log m)$$
Therefore the probability that there is a choice of $a_1, ..., a_n$ creating such a facet is at most 
$$m^n m^{-C_4C_1},$$ 
set $C_1$ large enough that this is $o(1/m^n)$, and this will be our bound of the probability that (1) fails to hold.\\

Now we turn our attention to (2). For (2) it will suffice to show that no vertex of $P$ is contained in more that $\sqrt[n]{\Delta}$ edges. Here we will make use of (1). If all the edges of $P$ have length less than $\epsilon$, then for any vertex $a_i$, all the neighboring vertices must lie in the spherical cap of height $\frac{\epsilon^2}{2}$. Thus the number of neighbors of $a_i$ is bounded by the number of vertices in this spherical cap. The probability that a fixed vertex $a_i$ lies in this spherical cap is at most $C_4 \left(\frac{\log m}{m}\right)$. Moreover the expected number of vertices in this cap is stochastically dominated by a binomial random variable with $m - 1$ trials and success probability $C_4 \left(\frac{\log m}{m}\right)$.  This has expectation $(1 - o(1))C_4 \log m$. Thus for $C_2$ sufficiently large, by Chernoff bound the probability that the degree of $a_1$ exceeds $C_2 \log m$ is $o(1/m^{n + 1})$. Thus the probability that some vertex has large degree is $o(m^{-n})$. Now the probability that (2) holds is at most
$$\Pr((2) \text{ holds} \mid (1)\text{ holds})\Pr((1)\text{ holds}) + \Pr((1)\text{ fails to hold}), $$
and we have showed that both of these terms are $o(m^{-n})$. \\

Finally we are ready to count the expected number of vertices which have a coface that contribute an edge to $P'$ when (1) and (2) hold. Fix a vertex $a_1$ and assume that (1) and (2) hold, then the distance from $a_1$ to the plane of intersection is at most $\epsilon$ if $a_1$ is to contribute an edge to $P'$. Indeed if the distance from the plane of intersection to $a_1$ is larger than $\epsilon$, but $a_1$ nonetheless contributes an edge to $P'$, then some facet containing $a_1$ intersects $\Span(u, v)$. Call such a facet $\sigma$ and let $x$ be a point of $\sigma \cap \Span(u, v)$. By (1), the distance from $a_1$ to $x$ is at most $\epsilon$ as all vertices of $\sigma$ live in the ball of radius $\epsilon$ around $a_1$, but $x$ is in the intersection plane, contradicting the assumed distance from $a_1$ to the plane. Thus we bound the number of vertices close to the intersection plane. \\

If $x = (x_1, ..., x_n)$ is a point inside the unit sphere then the distance squared from the intersection plane to $x$ is exactly $x_3^2 + \cdots + x_n^2$. If this is at most $\epsilon^2$, then $x_1^2 + x_2^2 \geq 1 - \epsilon^2$. So we bound the probability that a randomly chosen point on the unit sphere satisfies $x_1^2 + x_2^2 \geq 1 - \epsilon^2$. This is at most $C_5 \epsilon^{n-2} = O(m^{-(n - 2)/(n - 1)} \log m)$ (see Lemma \ref{beltlemma} of the appendix for details). Thus the expected number of such vertices given that (1) and (2) hold is at most $m^{1/(n - 1)} \log m$. And in that case, the number of facets contributing an edge to $P'$ is at most $m^{1/(n - 1)} \log m \Delta = \tilde{O}(m^{1/(n - 1)})$. On the other hand, the probability that (1) and (2) do not both hold is $o(m^{-n})$ and in this case the number of edges in $P'$ is trivially upper bounded by $m^n$, so the expectation when (1) and (2) do not both hold is $o(1)$. This finishes the proof.
\end{proof}

We are now ready to describe the full shadow vertex algorithm and then to give a new proof of an upper bound on its asymptotic average complexity. We present the dual description of the algorithm as is described in Borgwardt.  Given point $a_1, ..., a_m$ determining $Y := CH(0, a_1, ..., a_m)$ in $\R^n$, and a vector $v$ the shadow-vertex algorithm will find the facet of the convex hull intersected by $v$, or it will determine that no such facet exists. The algorithm proceeds in stages by dimension from 2 to $n$. At dimension $k$ the algorithm considers the projected linear programming problem in $\R^k$ given by projecting each $a_i$ and $v$ into $k$-dimensions by omitting the final $n - k$ coordinates of each. Following the notation of Borgwardt, let $\Pi_k$ denote this projection map from $\R^n$ to $\R^k$. The algorithm is as follows.
\begin{enumerate}
\item Apply the standard simplex algorithm to $\Pi_2(Y)$ with objective function $\Pi_2(v)$ to find an edge $CH(a_1, a_2)$ which is intersected by $\R^+\Pi_2(v)$. If such an edge does not exist, then there is no solution to the original problem, and the algorithm reports that is the case and STOPS.
\item Set $k = 2$
\item  While $k < n$ do
\item Set $k = k + 1$
\item Take the optimal simplex $CH(a_1, ..., a_{k - 1})$ found previously, this will correspond to a $(k - 2)$-dimensional face of $\Pi_k(Y)$. Find a boundary simplex $CH(a_1, .., a_{k - 1}, a_k^*)$ which belongs to $\Pi_k(Y)$. 
\item Let $C_k$ denote the polygon given by intersection $\Pi_k(Y)$ with the span of $\Pi_k(e_k)$ and $\Pi_k(v)$. Starting from the edge of $C_k$ determined by $CH(a_1, ..., a_{k - 1}, a_k^*)$, take pivot steps in $C_k$ to find an edge intersected by $\R^+\Pi_k(v)$.  If no such edge exists, report no solution to the original problem and STOP. Otherwise, this edge corresponds to a facet of $\Pi_k(Y)$ intersected by $\R^+ \Pi_k(v)$. This will be the optimal facet for the next step.
\item RETURN TO START OF LOOP.
\item Return $CH(a_1, ..., a_n)$. 
\end{enumerate}

Regarding the complexity, it is determined by the combinatorical complexity of each $C_k$ and of the first projection into $\Pi_k(Y)$. Notice that each $C_k$ is given by a projection to $\Pi_k(Y)$ and then an intersection with a certain plane determined by the random objective function $v$. In particular $C_k$ is not given directly by intersecting its defining 2-dimensional plane with $Y$, so we cannot directly apply Lemma \ref{shcomplexity}, we have to consider how the intermediate projection step affects the complexity of the resulting shadow polygon.\\

It turns out however, that this intermediate projection does not affect the complexity too much. If we generate $v$ randomly at the start then we have finitely many projection planes and intersection planes that are determined by $v$ (and a standard basis vector) or are fixed, thus by nondegeneracy we can consider one such plane at a time. If we have $Y$ and $k$, then $C_k$ is given by projection onto $e_1, ..., e_k$ followed by intersection with the plane given by $\Pi_k(v)$ and $e_k$, two vectors which trivially lie in $\Span\{e_1, ..., e_k\}$.  Moreover while $\Pi_k(v)$ and $e_k$ determine an intersection and not a projection, the analysis of the complexity in Lemma \ref{shcomplexity} is totally determined by the complexity of the corresponding \emph{projection}. Indeed counting the vertices in this projection is exactly what we did for the proof of Lemma \ref{shcomplexity} (see Figure 1). Thus the complexity of $C_k$ is bounded by the complexity of two sequential projections of $Y$, first onto $e_1, ..., e_k$ and then to $\Pi_k(v)$, $e_k$. But this is just the same as a single projection onto $\Pi_k(v), e_k$. Thus the complexity of $C_k$ is always at most $\tilde{O}(m^{1/(n - 1)})$ (and the same holds for the two dimensional step as we consider projection onto $e_1, e_2$).  \\

Since $n$ is fixed and $n$ determines the number of times the WHILE loop is repeated, we have the total complexity is $\tilde{O}(m^{1/(n - 1)})$ on average, immediately proving Theorem \ref{asymptoticshadowcomplexity}. \\

We have omitted the proof that the shadow-vertex algorithm solves the linear programming problem, but we refer the reader to Chapters 0 and 1 of  \cite{Borgwardt} for those details. What we have done, however, is given a shorter argument to explain the average complexity, which in the case of fixed $n$, still gives the right exponent on $m$.\\

\begin{remark}
One motivation to study the average-complexity of the simplex method is that in practice the simplex method works well, however \cite{KleeMinty} show that in the worst case the complexity of the simplex method is exponential. On the other hand a random model like $P \sim P(n,m)$ will likely not reflect a typical linear programming problem. In recent years \emph{smoothed analysis} has been considered as a mix of worse-case and average-case to try and better understand what happens in a setting more typical than one finds in a purely random setting. For more on smoothed analysis as it relates the simplex method we refer the reader to \cite{HowFast, DH, DS, ST}. 
\end{remark}

\begin{remark}
Establishing complexity results on linear programming is also related to Smale's ninth problem found in \cite{Smale} on determining the complexity of linear programming. In particular, Smale asks whether or not there is a \emph{strongly-polynomial time algorithm} for solving the linear programming problem. The random setting is a special case of this problem, and as Smale's ninth problem remains open any better understanding of the random case could help to solve this problem.
\end{remark}
\section{Computing the double description of a random polytope}
Convex hull algorithms also work well on random polytopes. This is discussed in, for example \cite{GiftWrapping, BeneathBeyond}. That is, given the vertices of a random polytope it is typically easy to find its hyperplane description. Indeed for any rotationally symmetric model with $m$ vertices chosen in $\R^n$, \cite{BeneathBeyond} gives very precise bounds in terms of both $n$ and $m$ on the expected complexity of computing the convex hull of the resulting polytope. In particular, the result of \cite{BeneathBeyond} implies that for $P(n, m)$ with $n$ fixed and $m \rightarrow \infty$, the expected complexity of the Beneath-Beyond algorithm is $O(m^2)$. Here we will give a proof, avoiding the integrals necessary to prove the precise statements in \cite{BeneathBeyond}, that the expected complexity of the Beneath-Beyond algorithm is $\tilde{O}(m^2)$ for $n$ fixed and $m$ tending to infinity. Once again we want the proof to come directly from the geometry of the sphere. \\

The Beneath-Beyond Algorithm is an incremental algorithm for computing the double description of a polytope (that is for describing its facets from its vertices or vice versa). It has been described and rediscovered several times, \cite{BeneathBeyond} credits it to \cite{Grunbaum} and \cite{ABS} credits it to \cite{Seidel} while pointing out that the first incremental algorithm for computing the double description of a polytope comes from \cite{MRTT}. For detailed discussion of the Beneath-Beyond algorithm we refer the reader to \cite{ABS, Edelsbrunner, Grunbaum, Joswig}.  For our discussion here we work with the primal interpretation of the Beneath-Beyond Algorithm, that is we assume that the vertices of a polytope $P$ are given and we compute the facets of $P$. This is done sequentially. We fix an order $a_1, a_2, ..., a_m$ on the vertices of $P$ and at step $t$ we have computed the facets of the convex hull of $0, a_1, ..., a_t$, we denote this polytope as $P_t$. Let $b_1, ..., b_l$ denote the vectors describing the facets of $P_t$. Explictly a point in $x \in R^n$ lies in the interior of $P_t$ if and only if $\langle b_i, x \rangle \leq 1$ for all $i \in \{1, ..., l\}$. The task at each step becomes to update the list of facets after adding $a_{t + 1}$. For $a_{t + 1}$ we check for each $b_i$ whether or not $a_{t + 1}$ lies on the same side of $b_i$ as the origin or not. If $a_{t + 1}$ lies on the same side of $b_i$ as the origin then we keep $b_i$ as a facet of $P_{t + 1}$, otherwise remove $b_i$ and add as new facets all those facet given by the cone an $(n - 2)$-dimensional facet of $b_i$ with cone point $a_{t + 1}$. Note that for this last step we should keep track of the $(n-2)$-facets of each facet we add along the way. Now if $P_t$ is a simplicial polytope we have $l$ linear inequalities to check to decide if each face belong to $P_{t + 1}$ and for those that do we have $n$ new faces to add. \\

In general, the Beneath-Beyond algorithm will produce a triangulation of the given polytope, and determining the actual facets from the triangulation requires a bit more work as discussed in \cite{ABS, Joswig}. In our setting, however, $P \sim P(n, m)$ will already be a simplicial polytope so finding the triangulation is equivalent to find the facets of the polytope, and moreover the polytopes at the intermediate steps $P_t$ will be distributed as $P(n, t)$. Intuitively we should expect the complexity of the Beneath-Beyond algorithm to be $\tilde{O}(m^2)$ because at step $t$ we expect $P_t$ to have $\tilde{O}(t)$ faces so the complexity of the algorithm is 
$$\sum_{t = 1}^{m} \tilde{O}(t) = \tilde{O}(m^2).$$
Indeed we make this precise here
\begin{theorem}
Fix $n \geq 2$, the expected complexity of the Beneath-Beyond Algorithm with input $P \sim P(n, m)$ is $\tilde{O}(m^2)$. 
\end{theorem}
\begin{proof}
We first show that with high probability $P_t \sim P(n, t)$ has fewer than $O(m \log^{n - 1} m)$ facets for every $t < m$. Of course, for $t < \sqrt[n] m$, this holds. Otherwise, for any $t < m$ by Theorem \ref{concentration} there exists a constant $C$ depending only on $n$ so that the probability that $P_t \sim P(n, t)$ has more than $C m \log^{n - 1} m$ facets is at most $\frac{C}{t^{n^2 + 3n}}$. Thus the probability that there exists a $t$ so that the number of faces of $P_t$ exceeds this bound is at most
$$\sum_{t = \sqrt[n] m}^{m} \frac{C}{t^{n^2 + 3n}} \leq m \frac{C}{m^{n + 3}} = m^{-(n + 2)}$$
Thus with probability $m^{-(n + 2)}$ the Beyond-Beneath algorithm takes $O(m^2 \log^{n - 1} m)$ steps to complete. Thus the expectation is $\tilde{O}(m^2)$ since we may use the trivial upper bound of $m^{n + 1}$ for the contribution  the expectation from the $O(1/m^{n + 2})$ proportion of experiments where the complexity of one of the preliminary polytopes exceeds $C m \log^{n - 1} m$. 
\end{proof}

\section{Concluding Remarks}
Here we have given new proofs of the following facts about $P \sim P(n, m)$ in the asymptotic case. 
\begin{itemize}
\item The expected number of facets of $P \sim P(n, m)$ is $\tilde{O}(m)$.
\item The expected complexity of the shadow-vertex algorithm applied to $P \sim P(n, m)$ to solve a linear programming problem with random linear objective function is $\tilde{O}(m^{1/(n - 1)})$.
\item The expected complexity of the beyond-beneath algorithm to find the hyperplane description of $P \sim P(n, m)$ is $\tilde{O}(m^2)$. 
\end{itemize}
While we have extra log factors in all of these upper bounds and the constants are far from best possible, we nonetheless explain without the integrals required for more refined results as in \cite{Borgwardt, BeneathBeyond, BMT} why the exponents are what they are in the more precise version. Essentially, the general paradigm for our proof strategy comes from the fact that the Hausdorff distance from $P$ to $S^{n - 1}$ is controlled by some function $f(m)$ tending to zero at an easily computable rate as $m$ tends to infinity and then there simply aren't enough hyperplanes far from the origin in the random model, so even if they all contribute a facet, there are still only very few facets. \\

We can also describe the paradigm in the language of Vietoris--Rips complexes. Given a point cloud $K$ in $\R^n$ and a distance $\epsilon > 0$, the Vietoris--Rips complex on $K$ with radius $\epsilon$ is defined to be the simplicial complex on $K$ where $\{a_1, ..., a_l \} \subseteq K$ is a face of the complex if and only if for every $i, j \in 1, ..., l$ the distance from $a_i$ to $a_j$ is less than $\epsilon$. In the present setting we have some $\epsilon$ coming from the upper bound on the Hausdorff distance from $P$ to $S^{n - 1}$ and then rather than compute the complexity of $P$ directly we compute the complexity of the Vietoris--Rips complex with vertices given as the random points and radius $\epsilon$. By what we know about the Hausdorff distance the facets of $P$ sit inside this Vietoris--Rips complex with high probability.

\section*{Appendix: Computation of relevant geometric measures}
Many of the proofs here rely on good upper and lower bounds for the Lebesgue measure of certain subsets of the $n$-dimensional sphere or the $n$-dimensional ball. In this appendix, we compute these upper and lower bounds. Doing so requires integration of course, but these integrals are elementary and not like those present when proving the state-of-the-art results on random polytopes. \\

Given the $n$-dimensional ball and $r, h \in (0, 1)$, we are interested in the following two subsets:
$$C_n(h) := \{(x_1, x_2, ..., x_n) \in B_n(1) \mid x_n \geq 1 - h\}$$
$$L_n(r) := \{(x_1, x_2, ..., x_n) \in B_n(1) \mid x_1^2 + x_2^2 \geq r^2\}$$
The first is a spherical cap of height $h$, and we will call the second a spherical belt of radius $r$. We are primarily interested in the $(n - 1)$-dimensional measure of the intersection of each of these with the boundary of $B_n(1)$, however these measures are easier to derive after finding the $n$-dimensional Lebsegue measure of the two sets as given. Drawing the natural analogy to the $n = 3$ case we use \emph{volume} of $C_n(h)$ to be its $n$-dimensional Lebesgue measure and use \emph{surface area} of $C_n(h)$ to be the $(n - 1)$-dimensional Lebesgue measure of $C_n(h)$ intersected with the boundary of $B_n(1)$. Similarly, we define the volume and surface area of $L_n(r)$. Moreover we let $v_n$ denote the volume of $B_n(1)$ and $s_n$ denote its surface area. We now compute these four quantities.
\begin{lemma}\label{beltlemma}
Fix $r \in (0, 1)$, then the volume of $L_n(r)$ is $v_n(1 - r^2)^{n/2}$ and its surface area is $s_n(1 - r^2)^{(n - 2)/2}$.
\end{lemma}
\begin{proof}
We begin with the volume. It is straightforward to see that if $x_1^2 + x_2^2 = t^2$ then the set of points in $B_n(1)$ that map to $(x_1, x_2)$ forms an $(n - 2)$-dimensional ball of radius $\sqrt{1 - t^2}$. From this we have that the following integral computes $L_n(r)$
$$L_n(r) = \int_{\{(x_1, x_2) \in \Omega_2 \mid x_1^2 + x_2^2 \geq r^2\}} \sqrt{1 - (x_1^2 + x_2^2)}^{n - 2} v_{n - 2} dx_1dx_2$$
By a change of variables to polar coordinates this becomes
$$L_n(r) = \int_r^1 \int_0^{2 \pi} \sqrt{1 - t^2}^{n - 2} v_{n - 2} t dt d\theta$$
This integral evaluates to 
$$L_n(r) = \frac{2 \pi v_{n - 2}}{n} \sqrt{1 - r^2}^{n} = v_n (1 - r^2)^{n/2}$$
Note that here we have used the well-known recurrence 
$$v_n = \frac{2\pi v_{n - 2}}{n}.$$
In fact one could \emph{derive} this recurrence by evaluating the above integral for $r = 0$. \\

Now with the volume in hand we compute the surface area by taking a limit. Naturally, if one considers the taking a uniform random point from the unit sphere in $\R^n$, then the probability that it lies on the surface of $L_n(r)$ is given by the following limit
\begin{eqnarray*}
\frac{SA(L_n(r))}{s_n} = \lim_{R \rightarrow 1} \frac{v_n(1 - r^2)^{n/2} - v_n(1 - r^2/R^2)^{n/2} r^n}{v_n - R^n v_n}
\end{eqnarray*}
By a routine application of l'Hopital's rule, this limit is equal to the following limit
\begin{eqnarray*}
&&\lim_{R \rightarrow 1} \frac{v_n \left[(1 - r^2/R^2)^{n/2} n R^{n - 1} + R^n\frac{n}{2}(1 - r^2/R^2)^{(n - 2)/2} \frac{2r^2}{R^3}\right]}{nR^{n - 1}v_n}\\
&=& \lim_{R \rightarrow 1} \left((1 - r^2/R^2)^{n/2} + \frac{1}{R^2} (1 - r^2/R^2)^{(n - 2)/2} r^2\right) \\
&=& (1 - r^2)^{(n - 2)/2}
\end{eqnarray*}
Thus the surface area of $L_n(r)$ is $s_n(1 - r^2)^{(n - 2)/2}$.
\end{proof}
So for the belt, the formulas for surface area and volume are easy and can be computed exactly for any $r$. For $C_n(h)$, we won't obtain such nice formulas. However, we will derive a formula that we can approximate very well for values of $h$ close to zero, and these are the values of $h$ we are interested in for our proofs. 
\begin{lemma}\label{capvolume}
Fix $h \in (0, 1)$ then the volume of $C_n(h)$ is 
$$V(C_n(h)) = \int_{1 - h}^1 (1 - y^2)^{(n - 1)/2} v_{n - 1} dy,$$ and the surface area is 

$$SA(C_n(h)) = s_n\left(\frac{\sqrt{2h - h^2}^{n - 1} v_{n - 1}}{nv_n}(1- h) + \frac{V(C_n(h))}{v_n}\right)$$
\end{lemma}
\begin{proof}
The formula for the volume is immediate, indeed if $x_n = y$ then the intersection of that particular hyperplane with $B_n(1)$ is an $(n - 1)$-dimensional ball of radius $\sqrt{1 - y^2}$. As in the case of $L_n(r)$, we can compute $SA(C_n(h))$ as the following limit. 
$$\frac{SA(C_n(h))}{s_n} = \lim_{r \rightarrow 1} \frac{V_n(h) - r^nV(C_n(h))(1 - \frac{1 - h}{r})}{v_n - r^n v_n}$$
 In order to simplify the notation we let $V_n(h)$ denote $V(C_n(h))$ as a function of $h$ and its derivative is
$$V'_n(h) = \sqrt{2h - h^2}^{n - 1} v_{n - 1}$$
We now compute the above limit using l'Hopitals rule.
\begin{eqnarray*}
\frac{SA_n(h)}{s_n} &=& \lim_{r \rightarrow 1} \frac{d(V_n(h) - r^nV_n(1 - \frac{1 - h}{r}))/dr}{d(v_n - r^n v_n)/dr} \\
&=& \lim_{r \rightarrow 1} \frac{r^n V'_n(1 - \frac{1 - h}{r}) \frac{1 - h}{r^2} + nr^{n - 1} V_n(1 - \frac{1 - h}{r})}{nr^{n - 1}v_n} \\
&=& \lim_{r \rightarrow 1} \frac{ V'_n(1 - \frac{1 - h}{r}) (1 - h) + nr V_n(1 - \frac{1 - h}{r})}{nrv_n} \\
&=& \frac{V_n'(h)(1 - h) + nV_n(h)}{nv_n} \\
&=& \frac{\sqrt{2h - h^2}^{n - 1} v_{n - 1}}{nv_n}(1- h) + \frac{V_n(h)}{v_n}
\end{eqnarray*}
\end{proof}
As stated, these formulas for volume and surface area of $C_n(h)$ do not appear to be particularly useful. However, we are really only interested in behavior for $n$ fixed and $h \rightarrow 0$. Here we make the following observation. 
\begin{lemma}
As $h \rightarrow 0$, one has
$$V(C_n(h)) \sim \frac{v_{n - 1}\sqrt{2h - h^2}^{n + 1}}{2(n + 1)}$$
\end{lemma}
\begin{proof}
Let $\delta > 0$ be given. For any $h$, we have an integral formula for $V(C_n(h))$ from Lemma \ref{capvolume}. Now for $h < \delta$ we have the following upper and lower bound on the $V(C_n(h))$
$$\int_{1 - h}^1 \sqrt{1 - y^2}^{n - 1} v_{n - 1} \frac{y}{1 - \delta} dy \geq V(C_n(h)) \geq \int_{1 - h}^1 \sqrt{1 - y^2}^{n - 1} v_{n - 1} y dy $$
Thus evaluating the simple integrals on either side we have that $V_n(h)$ is bounded as:
$$\frac{v_{n - 1}\sqrt{2h - h^2}^{n + 1}}{2(n + 1)(1 - \delta)} \geq V(C_n(h)) \geq \frac{v_{n - 1}\sqrt{2h - h^2}^{n + 1}}{2(n + 1)}.$$
\end{proof}
\begin{corollary}\label{geometrycorollary}
For any $n$, there exists constants $C_1$ and $C_2$ so that as $h \rightarrow 0$,
$$V(C_n(h)) \sim C_1 h^{(n + 1)/2},$$
and
$$SA(C_n(h)) \sim C_2 h^{(n - 1)/2}$$
\end{corollary}
In the proof of Theorem \ref{spherecomplexity} we used the expected Hausdorff distance to upper bound the number of facets containing a fixed vertex. In doing so we used the following lemma which we prove now. 
\begin{lemma}
If $P$ is an $n$-dimensional polytope with all vertices on the boundary of the $n$-dimensional unit ball and so that the Hausdorff distance to the unit sphere is at most $\delta$, then for any vertex $v$ of $P$, all facets containing $v$ must have all of their vertices inside the spherical cap centered at $v$ of height at $4 \delta$. 
\end{lemma}
\begin{proof}
Fix $v$ and suppose that $(v, v_2, ..., v_n)$ is a facet of $P$ with at least one vertex outside the cap of height $4 \delta$ centered at $v$. Call such a vertex $w$. We now may restrict to the two dimensional picture given by the span of $v$ and $w$ in $\R^n$. Let $h$ denote the distance from the origin to the midpoint of the edge $(v, w)$ and let $\epsilon$ denote the length of $(v, w)$. Moreover let $u$ be the endpoint of the vector $w$ projected onto $v$ and let $t = \langle w, v \rangle$ be the distance from the origin to $u$. By assumption $1 - t > 4 \delta$, but $h \geq 1 - \delta$, we will show that this is impossible.  We have two right triangles, one given by $v$, $u$, and $w$ and one by $0$, $u$, $w$. Thus using the shared edge, we have by the Pythagorean theorem
$$\epsilon^2 - (1 - t)^2 = 1 - t^2$$
\end{proof}
Therefore
$$\epsilon^2 = 2(1 - t).$$
On the other hand we have a right triangle given by the origin, $w$ and the midpoint of $(w, v)$. Thus we have
$$(\epsilon/2)^2 + h^2 = 1.$$
It follows that
$$(1 - t) = 2(1 - h^2) = 2(1 + h)(1 - h) \leq 4(1 - h) \leq 4 \delta.$$
This completes the proof since we assumed that $1 - t$ was more than $4 \delta$.

\section*{Acknowledgments}
The author thanks Michael Joswig for suggesting the project of establishing more geometric proofs of results on random polytopes that led to this article and for several helpful discussions during the research and writing process. 

The author also gratefully acknowledges funding by Deutsche Forschungsgemeinschaft (DFG, German Research 
Foundation) Graduiertenkolleg 2434 "Facets of Complexity".

\bibliography{ResearchBibliography}
\bibliographystyle{amsplain}

\end{document}